\documentclass[12pt,english]{article}
\usepackage{amsmath,amssymb,amsthm,bbm,float,graphicx,geometry,lmodern,mathtools,parskip,setspace,subcaption,subcaption}
\usepackage[colorlinks]{hyperref}
\newcommand{\edit}[1]{\textcolor{black}{#1}}

\allowdisplaybreaks

\usepackage{caption}
\DeclarePairedDelimiter{\set}{\lbrace}{\rbrace}
\newcommand{\N}{\mathbb{N}}
\newcommand{\E}{\mathbb{E}}
\newcommand{\R}{\mathbb{R}}

\newtheorem{thm}{Theorem}[section]

\newtheorem{lem}[thm]{Lemma}

\newtheorem{prop}[thm]{Proposition}
\author{
Arne Grauer\thanks{Department of Mathematics and Computer Science, University of Cologne, Weyertal 86-90, 50931 K\"oln, Germany.}\\AGrauer@math.uni-koeln.de
\and
Lukas L\"{u}chtrath\footnotemark[1]\\L.Luechtrath@uni-koeln.de
\and
Mark Yarrow\thanks{School of Mathematics and Statistics, University of Sheffield, Hicks Building, S3 7RH.}\footnotemark[2]\\ Mark.Yarrow@sheffield.ac.uk
}
\title{Preferential attachment with location-based choice: Degree distribution in the noncondensation phase}
\usepackage{etoolbox}
\ifundef{\abstract}{}{\patchcmd{\abstract}%
{\quotation}{\quotation\noindent\ignorespaces}{}{}}
\begin{document}
\maketitle 
\begin{abstract}
We consider the preferential attachment model with location-based choice introduced by Haslegrave, Jordan and Yarrow as a model in which condensation phenomena can occur \cite{Sou.2}. In this model, each vertex carries an independent and uniformly distributed location. Starting from an initial tree, the model evolves in discrete time. At every time step, a new vertex is added to the tree by selecting $r$ candidate vertices from the graph with replacement according to a sampling probability proportional to these vertices' degrees. The new vertex then connects to one of the candidates according to a given probability associated to the ranking of their locations. In this paper, we introduce a function that describes the phase transition when condensation can occur. Considering the noncondensation phase, we use stochastic approximation methods to investigate bounds for the (asymptotic) proportion of vertices inside a given interval of a given maximum degree. We use these bounds to observe a power law for the asymptotic degree distribution described by the aforementioned function. Hence, this function fully \edit{characterises} the properties we are interested in. The power law exponent takes the critical value one at the phase transition between the condensation - noncondensation phase.
\\
\\
\footnotesize{{\textbf{AMS-MSC 2010}: 05C80}}
\end{abstract}
\newpage
\section{Introduction}\label{SecIntro}
The study of complex networks is a prevalent area of interest for researchers as many seemingly dissimilar structures observable in the real world can be modelled using \edit{a common set of} techniques. This is due to many large networks sharing similar topological properties. \edit{For instance, it has been observed that the empirical degree distribution of many large-scale real world networks follows an approximate \emph{power law} over a large finite range of degrees. Hence, we seek families of models that imitate this behaviour.} 

\edit{A probabilistic approach is to build networks as a growing sequence of graphs in which the degree distribution follows a power-law when the number of vertices is going to infinity. That is, the tail of the asymptotic proportion of vertices of degree at least $k$ behaves like $k^{-\tau}$ for some power-law exponent $\tau$. We call such a network \emph{scale-free}.} 

In 1999 Barab\'asi and Albert popularised \emph{preferential attachment} \cite{P.10} as a method of growth which utilises the famous \emph{rich get richer} concept. As a new vertex joins the network, it forms an edge to already existing vertices with probability proportional to the degrees of current vertices. This mechanism was generalised by Dorogovtsev et al \cite{P.42} by biasing the selection mechanism to enhance or suppress the influence of the degrees. It was shown by various authors that this \edit{building} mechanism indeed leads to scale-free networks \cite{P.36, P.42, P.24}. 
Although preferential attachment is often an accurate method of modelling \edit{scale-free} networks, it fails to consider a new vertex's potential to attract new edges. In order to tackle this issue, Bianconi and Barab\'asi \cite{P.1} suggested the addition of \emph{vertex fitness} as an additional parameter. Here, each vertex joins the network with its own randomly chosen fitness, allowing for a new level of competition between vertices, separate from their current edge-based popularity. Many models have been devised which include this `attractiveness' coefficient, most notably by Borgs et al.\ \cite{P.2} and Dereich and Ortgiese \cite{P.4}. Another way of incorporating a vertex's inherent potential for growth is by introducing the notion of \emph{choice}. In \cite{P.6,P.8,P.7} preferential attachment is used to sample a set of vertices from the network as candidates for connection. Afterwards, a preassigned attachment rule based on the degrees of the sampled vertices is used to decide where new edges are formed. 

Furthermore, a feature of interest is the \emph{condensation} \edit{phenomenon}. Condensation occurs if the total degree of an $o(n)$ subset of vertices grows linearly in time $n$. Loosely speaking, at any time there exists some vertex whose degree \edit{dramatically} dominates the others. Whereas in classical preferential attachment condensation cannot occur, it was shown that both preferential attachment with choice and models with fitness can \edit{exhibit} condensation \cite{P.2,P.15,P.3, P.4,P.5}. 

In this paper, we consider the preferential attachment with location-based choice model \edit{introduced} by Haslegrave, Jordan and Yarrow in \cite{Sou.2} which can be seen as a generalised variant of \cite{Sou.1}. This model combines the ideas of both fitness and choice in a natural way. Starting from an initial tree graph, at each time step a new vertex joins the graph and is assigned its own location which is uniformly chosen from $(0,1)$. When this vertex joins the network, a subset of $r$ \edit{neighbour candidates} is sampled with probability proportional to their degree plus some constant $\alpha$. The sampled vertices are ranked according to their locations. Following this, a single vertex from the sample is chosen for connection to the new vertex according to some probability measure $\Xi$. Here, $\Xi$ can be used to make different regions of $(0,1)$ more or less appealing and thus incorporates more flexibility than in previous models. As in \cite{Sou.2}, we refer to location \edit{as} opposed to fitness \edit{in order not} to give the false impression of preferring the `fittest' vertex. We could choose any \edit{continuous} distribution on the real line but we do not expect any changes in the results as the connection mechanism only depends on the ordering of the vertices' locations and not their actual value. Hence, \edit{there is no loss in generality by restricting the locations distribution to the uniform one on $(0,1)$}. This has the added benefit that it matches previous work our results build on. \edit{From \cite{Sou.2}, it can be derived} that there exists a critical value $\alpha_c$ such that condensation can occur if $\alpha<\alpha_c$.

In this article, we \edit{give a new description of $\alpha_c$} and study the degree-distribution of this model in the noncondensation regime. We show that in the noncondensation case, $\edit{\alpha\geq\alpha_c}$, the model is asymptotically scale-free with a heavy tailed degree distribution with power-law exponent $\tau=\frac{2+\alpha}{2+\alpha_c}$. Hence, the critical value $\alpha_c$ for the condensation phase transition matches the one for which the power-law exponent is large enough for the degree distribution's first moment to exist. This behaviour coincides with our understanding of condensation. In the condensation phase, with positive probability \edit{a proportionally small number of vertices dominate the others}. The noncondensation phase is `regular' in the sense that a typical vertex has finite expected degree. As the behaviour of the degree distribution dramatically changes between the two phases, we lose the finite moments at that phase transition, \edit{even though for $\alpha=\alpha_c$ the network is still scale-free with $\tau=1$}. The same behaviour can be observed in similar models with choice \cite{P.6}. Although a power-law distribution is what one would hope for in the considered regime, it is notable that this is not the case in the original preferential attachment model with choice of Malyshkin and Paquette for more than two options \cite{P.8}. 

To derive the degree distribution, we introduce a function $f$ on the location space $(0,1)$ depending only on $\Xi$ that plays a key role in understanding the influence of location on \edit{the degree of a vertex}. Given a vertex with location $x$, the expected probability of choosing that vertex with respect to $\Xi$, out of a sample containing this vertex and $r-1$ uniformly located vertices is given by $f(x)/r$. We show that the condensation phase transition as well as the power-law exponent can be derived from the maximum value of $f$. To get this, we determine the concrete degree distribution of a vertex at a given location whose tail behaviour follows a power-law distribution dependent on $f$ from which we derive the final result. The function $f$ hints at where to search for the high degree vertices. Specifically, the larger the values of $f$ in a specific region, the more likely we are to find high degree vertices there. The question of the degree distribution in the condensation phase is also of some interest but cannot be achieved with our methods since we rely on some continuity properties in our proof that are not fulfilled in the condensation regime. 

The paper is structured as follows: In Section~\ref{SecTheModel} we formally introduce the model and formulate the main theorem. Afterwards we recall the phase transition conditions determined in \cite{Sou.2}. We introduce formally the function $f$ and rewrite these conditions. In Section~\ref{SecDegreeDistr}, we use stochastic approximation methods to deduce bounds of the growth of the empirical degree distribution. We use these bounds to deduce the asymptotic degree distribution, proving the main theorem. In the last section, we show numerical results and simulations for some interesting and important choices of $\Xi$ underlining our understanding and results.

\section{Model description and main result}\label{SecTheModel}
Let $r\geq 2$ be an initial integer model parameter and let $\Xi$ be a probability measure on $\set{1,\dots,r}$. In the following, we treat $\Xi$ as a probability vector $(\Xi_1,\dots,\Xi_r)$. 
Furthermore, let $G_0$ be an initial tree graph on $n_0\geq 2$ vertices $\set{v_{1-n_0},\dots, v_0}$. Additionally, let each vertex $v_i$ in $G_0$ have its own location $x_i$ that is drawn independently and uniformly at random from $(0,1)$ and is therefore almost surely unique. 

At time $n+1$, a new vertex $v_{n+1}$ assigned its own location $x_{n+1}$, again drawn independently and uniformly at random from $(0,1)$ is added to the graph. Given $G_n$ \edit{and the locations of all its vertices}, we form the graph $G_{n+1}$ by connecting the new vertex $v_{n+1}$ by a single edge to a vertex in $G_n$. Note that this maintains the tree structure of the graph. The connection mechanism is as follows: First, we sample $r$ candidate vertices with replacement from $G_n$ according to preferential attachment, i.e.\ proportional to the vertices' degrees plus a fixed constant $\alpha$. Second, $v_{n+1}$ chooses one vertex for connection out of the sample \edit{according to $\Xi$ applied to the ranks of the locations}. More precisely, fix $\alpha\in(-1,\infty)$ and denote $\deg_{G_n}(v_j)$ as the degree of vertex $v_j$ in $G_n$. We first select a sample of $r$ candidate vertices from $G_n$ with replacement \edit{so that independently for each of the $r$ candidates}
\begin{equation}\label{PA}
\mathbb{P}(v_i\edit{\text{ is sampled}\mid G_n})=\frac{\deg_{G_n}(v_i)+\alpha}{(n+n_0-1)(2+\alpha)+\alpha}.
\end{equation}
Here, due to the tree structure, the denominator equals the total degree weight of $G_n$, that is the sum over each vertices' degree plus $\alpha$. We next order the $r$ \edit{sampled vertices} according to their location. That is, we obtain a sample of vertices \edit{$\big(v^{(n+1)}_1,\dots, v^{(n+1)}_r\big)$} and associated locations \edit{$\big(x^{(n+1)}_1,\dots,x^{(n+1)}_r\big)$} such that the locations satisfy $x^{(n+1)}_1\leq\dots\leq x^{(n+1)}_r$. An important observation is that equality for the locations happens almost surely only if a vertex has been sampled multiple times. Thus, the ordered sample is uniquely determined. Finally, according to $\Xi$ one vertex out of \edit{$\big(v^{(n+1)}_1,\dots,v^{(n+1)}_r\big)$} is chosen for connection. That is, the probability that vertex $v^{(n+1)}_j$ is chosen for connection is given by $\Xi_j$. 

\subsection{Main Result} \label{SubSecMainResult}
As mentioned in the introduction, it is known that there exists a threshold $\alpha_c$ such that condensation can only occur if and only if $\alpha< \alpha_c$, see \cite{Sou.2} \edit{and Proposition~\ref{propAlphaCrit} below}. Let $\mu_k$ be the asymptotic proportion of vertices of degree at least $k$. 
\medskip
\begin{thm}\label{ThmPowerLaw} If $\alpha \edit{\geq} \alpha_c$, then \edit{$\mu_k$ exists and satisfies} 
\[\mu_k = k^{-\frac{2+\alpha}{2+\alpha_c}+o(1)},\]
as $k\to\infty$.
\end{thm}
\medskip
In order to prove this result, one has to understand the influence of the location on a vertex's degree. To this end, define $\Psi_n(x)$ as the \edit{conditional probability, given the graph $G_n$ and the locations of all the vertices of $G_n$,} that the new vertex $v_{n+1}$ selects under preferential attachment according to equation \eqref{PA} a vertex which has location at most $x$. Denote by $V(G_n)$ the vertex set of $G_n$. \edit{Then it holds that}
\begin{equation}\label{psinx}
\Psi_n(x)=\frac{1}{(n+n_0-1)(2+\alpha)+\alpha}\left(\sum_{v_i\in V(G_n):x_i\leq x}(\deg_{G_n}(v_i)+\alpha)\right). 
\end{equation}
The \edit{random} measures induced by $\Psi_n(x)$ converge weakly almost surely to a probability measure on $[0,1]$, whose continuous distribution function we call $\Psi(x)$ \cite[Theorem 2.2]{Sou.2}. \edit{Here, it is important to note that in general $\Psi(x)$ may be random. However, it is not random whenever $\alpha\geq\alpha_c$}. Finally, define the function $f:[0,1]\to \mathbb{R}_+$ by
\begin{equation}\label{connectprobwoPA}
f(x) = \sum^r_{s=1}s\Xi_s{r\choose s}x^{s-1}\left(1-x\right)^{r-s} ,
\end{equation}
(for more details about $\Psi_n, \Psi$ and $f$, we refer the reader to Section \ref{SubSecPhase}.) Conditioned on the event that there is a vertex at a given location $x$, we denote by $\nu(k, x)$ the probability that the vertex at location $x$ has asymptotically at least $k$ neighbours.
\begin{thm} \label{ThmLocalDegree} If $\alpha\edit{\geq}\alpha_c$ and \edit{$x\in(0,1)$}, \edit{then $\nu(k,x)$ is well-defined and satisfies}
\begin{equation*}
\nu(k,x)= k^{-\frac{2+\alpha}{f(\Psi(x))}+o(1)},
\end{equation*}
as $k\to\infty$.
\end{thm}
\begin{figure}
\begin{center}
\includegraphics[scale=0.3]{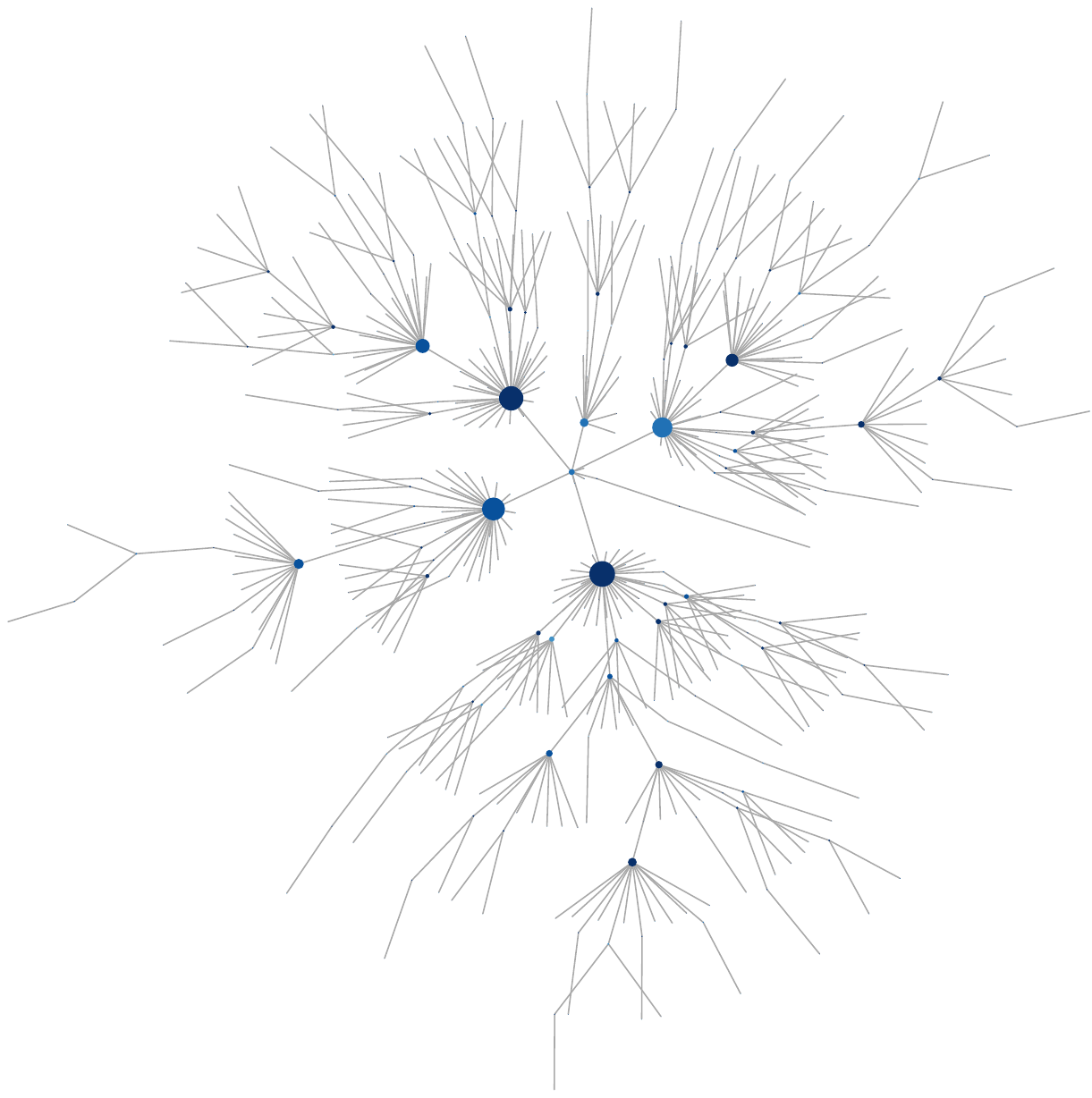}
\includegraphics[scale=0.3]{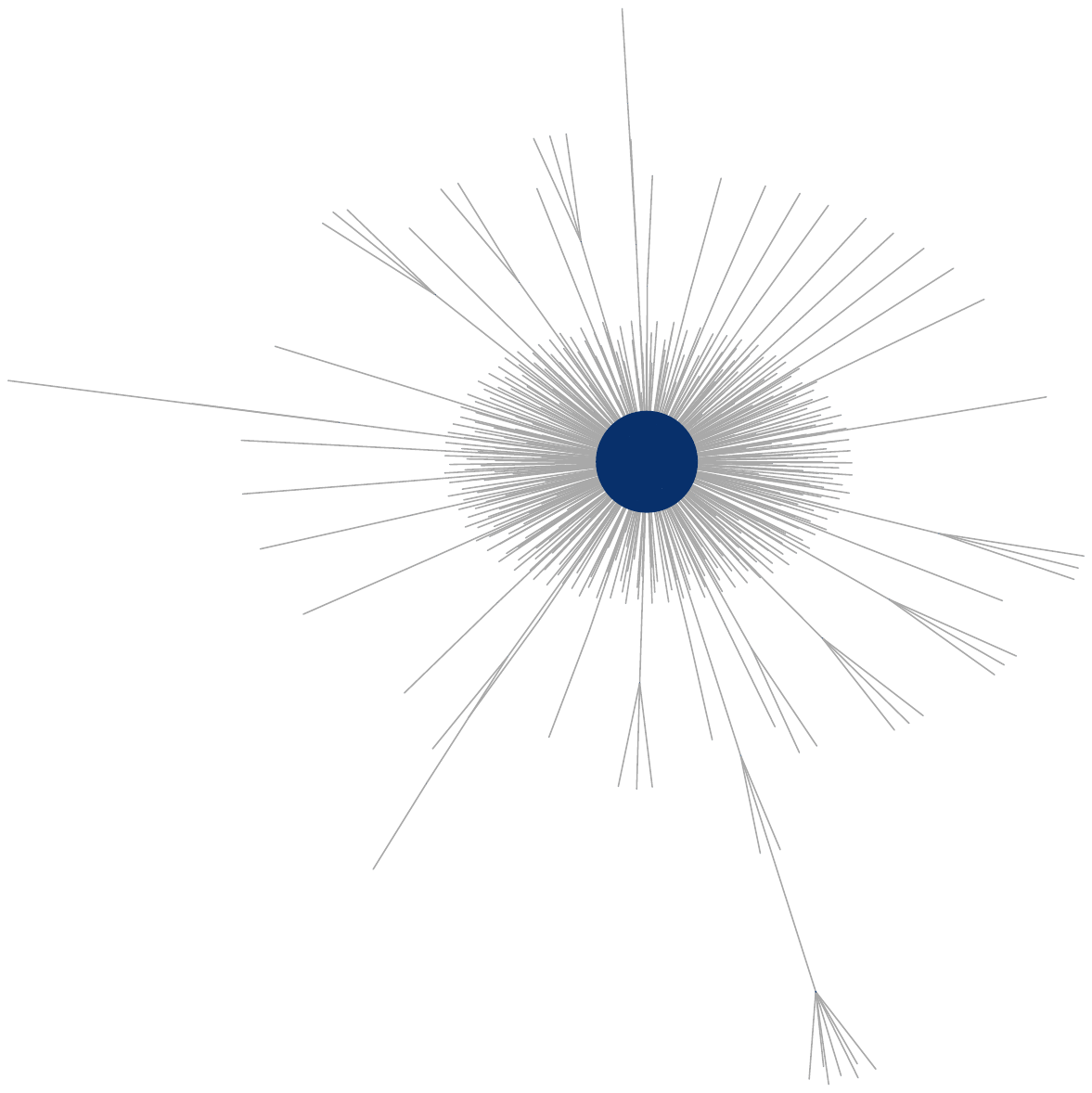}
\label{fig:Tree}
\caption{Plots of a simulated tree for $\Xi = (0,1,0)$ after $500$ vertices have been added. On the left, a realization for $\alpha>\alpha_c$ and on the right, a realization with $\alpha<\alpha_c$. In both cases, the start configuration consists of a root vertex and a single child, both with uniform drawn location. In the plot, the size of a vertex corresponds to its degree. We use colour saturation to indicate how close to the maximum value of $f$ a vertex's location is.}
\end{center}
\end{figure}

\subsection{Condensation phase transition}\label{SubSecPhase}
According to \eqref{psinx}, $\Psi_n(x)$ is almost surely monotonically increasing with $\Psi_n(0)=0$ and $\Psi_n(1)=1$. Hence, we can think of $\Psi_n(x)$ as a random distribution function on the location space. The measures induced by $\Psi_n(x)$ converge weakly almost surely to a (possibly random) probability measure on $[0,1]$. We call the distribution function of this limit $\Psi(x)$. 

\edit{We define condensation as a \emph{discontinuity} in $\Psi$ since a jumping point of $\Psi$ implies that $\Psi_n$ increases by $O(1)$ on an interval of length $o(1)$, as $n\to\infty$, matching the condensation description given in the introduction. Here, condensation may arise due to the existence of a persistent hub \cite[Theorem~2.3.]{Sou.2} as well as without a hub where the currently leading vertex is replaced over time \cite[Theorem~2.4.]{Sou.2}. The following proposition summarises arguments from \cite{Sou.2} showing that the discontinuity of $\Psi$ can only occur if $\alpha$ is smaller than the threshold $\alpha_c$. Additionally, it gives a new description of this threshold. We call $(-1,\alpha_c)$ the \emph{condensation} and $[\alpha_c,\infty)$ the \emph{noncondensation phase} of the model.} 

\edit{The function $f$ (see \eqref{connectprobwoPA}) only depends on the model parameter $\Xi$ and plays an important role in characterising the condensation phase transition.} It can be observed that $f$ is a probability density on $[0,1]$. Rewriting the binomial coefficient, one can interpret $f(x)/r$ as the expected probability of connecting with respect to $\Xi$ to a given vertex of location $x$ where the remaining $r-1$ vertices are chosen uniformly.
\edit{
\begin{prop}\label{propAlphaCrit}
There exists $\alpha_c\geq -1$ such that condensation can occur with a positive probability if $\alpha<\alpha_c$ but cannot occur if $\alpha\geq\alpha_c$. \\
Moreover, 
\begin{equation} 
\alpha_c = \max \set{f(y) : y\in [0,1]} - 2. \label{critalpha}
\end{equation}
\end{prop} }
\begin{proof}
As a function of $y\in[0,1]$, we define
\[
F_1(y;x,\Xi) = x(\alpha + 1) - (2+\alpha)y + \sum_{s=1}^r \Xi_s \sum_{i=s}^r \binom{r}{i} y^i (1-y)^{r-i}
\]
for $x\in [0,1]$. \edit{By \cite[Theorem 2.2]{Sou.2}, $\Psi_n(x)$ converges almost surely to a zero of the function $F_1(y;x,\Xi)$ and by \cite[Theorems 2.3, 2.4]{Sou.2} condensation} occurs with positive probability, whenever there exists $x\in (0,1)$ such that $F_1(y;x,\Xi)$ has a touchpoint. Here, we call $p\in (0,1)$ a touchpoint if $F_1(p;x,\Xi) = 0$ and there exists $\varepsilon>0$ such that either $F_1(y;x,\Xi) < 0$ for all $y\in (p-\varepsilon, p + \varepsilon)\backslash \set{p}$ or $F_1(y;x,\Xi) > 0$ for all $y\in (p-\varepsilon, p + \varepsilon)\backslash \set{p}$. Note that if $F_1(y;x,\Xi)$ is increasing in $y$ somewhere on $[0,1]$, one can vary $x$ in such a way that $F_1(y;x,\Xi)$ has a touchpoint. Hence, condensation can occur with positive probability for $\alpha < \alpha_c$, where
\begin{equation*}
\alpha_c = \inf \set{\alpha > -1 : F_1'(y;x,\Xi) \leq 0\ \text{for all}\ y\in (0,1)},
\end{equation*}
\edit{see also \cite[p. 792]{Sou.2}. Conversely,} if $\alpha \geq \alpha_c$, then, for all $x\in [0,1]$, $F_1(y;x,\Xi)$ has only one zero to which $\Psi_n(x)$ converges almost surely. Since $F_1$ is continuous \edit{and strictly decreasing in the neighbourhood of the root}, the zero $\Psi$ is continuous and almost surely no condensation can occur, \edit{proving the first part of the proposition. To prove \eqref{critalpha}, we calculate}
\begin{align*}
F_1'(y;x,\Xi) &= -(2+\alpha) + \sum_{s=1}^r \Xi_s \sum_{i=s}^r i \binom{r}{i} y^{i-1} (1-y)^{r-i} \left( 1 - \frac{r-i}{i} y (1-y)^{-1}\right)\\
&= -(2+\alpha) + \sum_{s=1}^r s \Xi_s \binom{r}{s} y^{s-1} (1-y)^{r-s} = - (2 + \alpha) + f(y). 
\end{align*}
\edit{Hence,} $F_1'(y;x,\Xi) \leq 0$ holds for all $y\in (0,1)$ if and only if $2+\alpha \geq \max\set{f(y): y\in [0,1]}$.
\end{proof}
\edit{Proposition \ref{propAlphaCrit} shows} that both preferential attachment and the location-based choice aspect are necessary for condensation to occur. Sampling according to preferential attachment but then choosing one vertex independently of the locations coincides with the choice of $\Xi_s=1/r$. Then, $f(x)\equiv 1$ and thus $\alpha_c=-1<\alpha$ for all $\alpha \in(-1,\infty)$. By Theorem~\ref{ThmPowerLaw}, the network is then scale-free with power-law exponent $\tau=2+\alpha\in(1,\infty)$, matching the results of \cite{P.42}. On the other hand, sampling without preferential attachment in this model coincides with the case $\alpha\to \infty$. Therefore, it holds $\alpha>\alpha_c$ for all choices of $\Xi$.
\edit{Summarizing the above, whenever $\alpha\geq \alpha_c$, no condensation can occur and the limiting distribution $\Psi$ is continuous and non-random. This is shown to be important in following sections.}

\section{Noncondensation phase degree distribution} \label{SecDegreeDistr}
\edit{We utilize a number of stochastic approximation techniques constructed by Robbins and Monro \cite{P.28} outlined in Pemantle \cite[Section~2]{P.9}. For a stochastic process $(X_n)_{n\in\N}\subset \R^n$ adapted to a filtration $(\mathcal{F}_n)_{n\in\N}$, the idea of stochastic approximation is to find a representation of the increments $X_{n+1}-X_n$ which fulfills certain properties. This then allows for results on the asymptotic behaviour of the process to be derived. Classically, we call equations of the form
\[X_{n+1}-X_n = \tfrac{1}{n}\big(F(X_n)+\xi_{n+1}+r_n\big)\]
stochastic approximation equations. Here, $F$ is an $\R^n$-vector field, $\xi_{n+1}$ is a noise term satisfying $\E[\xi_{n+1}\mid \mathcal{F}_n]=0$ and the remainder term $r_n$ is $\mathcal{F}_n$-measurable, converging to zero and satisfying $\sum_n n^{-1}|r_n|<\infty$ almost surely. Depending on the properties of $F$ and possible further assumptions on the noise $\xi_{n+1}$, different results for the asymptotic behaviour of the process are known \cite[Section~2]{P.9}. Many results can be further extended to hold, when $F$ is random, see e.g.\ \cite{ Sou.2, P.11}.} 

\edit{In our setting, we will need a statement for the asymptotic behaviour when only bounds on the increments are given. To this end, we adapt Lemma~5.4. of \cite{P.12} by Jordan and Wade.
}
\edit{
\begin{lem}\label{conflem}
Let $(\mathcal{F}_n)_{n\in\N_0}$ be a filtration. Furthermore, let $X=(X_n)_{n\in\N_0}$, $A_1=(A_1^{(n)})_{n\in\N_0}$, $A_2=(A_2^{(n)})_{n\in\N}$, $K_1=(K_1^{(n)})_{n\in\N_0}$, $K_2=(K_2^{(n)})_{n\in\N_0}$, $\xi=(\xi_n)_{n\in\N_0}$, $r_1=(r_1^{(n)})_{n\in\N_0}$ and $r_2=(r_2^{(n)})_{n\in\N_0}$ be real-valued stochastic processes adapted to $(\mathcal{F}_n)_{n\in\N_0}$ where $X$, $A_1$, $A_2$, $K_1$ and $K_2$ are non-negative and bounded. Let $(\gamma_n)_{n\in\N_0}$ be a sequence of non-negative constants and suppose that
\begin{align}
\gamma_n(A_1^{(n)}-K_1^{(n)}X_n+\xi_{n+1}+r_1^{(n)}) & \leq X_{n+1}-X_n \label{eqStochApprox} \\ & \leq \gamma_n(A_2^{(n)}-K_2^{(n)}X_n+\xi_{n+1}+r_2^{(n)}). \notag
\end{align} 
Assume further that
\begin{enumerate}
\item[(i)] $\E[\xi_{n+1}\mid \mathcal{F}_n] = 0$ and $\E[\xi_{n+1}^2\mid \mathcal{F}_n]\leq C$ for a finite constant $C$,
\item[(ii)] $\sum_{n=0}^\infty\gamma_n =\infty$, $\sum_{n=0}^\infty \gamma_n^2<\infty$ and $\sum_{n=0}^\infty |r_i^{(n)}|\gamma_n<\infty$ almost surely ($i=1,2$),
\item[(iii)] $0<\ell_i\leq K_i^{(n)}\leq u_i$ for some finite constants $\ell_i,u_i$ ($i=1,2$),
\item[(iv)] $\liminf\limits_{n\to\infty} \tfrac{A_1^{(n)}}{K_1^{(n)}}\geq L_1$ and $\limsup\limits_{n\to\infty} \tfrac{A_2^{(n)}}{K_2^{(n)}}\leq L_2$ almost surely.
\end{enumerate}
Then, almost surely, 
\begin{equation}
L_1\leq\liminf\limits_{n\to\infty} X_n\leq\limsup\limits_{n\to\infty} X_n \leq L_2.\notag
\end{equation}
\end{lem}
\begin{proof}
We only prove the lower bound for the $\liminf$ as the upper bound for the $\limsup$ works with analogous argumentation. For $\varepsilon>0$ there exists an almost surely finite $N_1$ such that $L_1\leq A_1^{(n)}/K_1^{(n)}+\varepsilon/2$ for $n\geq N_1$ by (iv). For each $x<L_1-\varepsilon$ and $n\geq N_1$, we have
\begin{align*}
A_1^{(n)}-K_1^{(n)}x\geq A_1^{(n)}-K_1^{(n)}(L_1-\varepsilon)\geq A_1^{(n)}-K_1^{(n)}\big(\tfrac{A_1^{(n)}}{K_1^{(n)}}-\tfrac{\varepsilon}{2}\big)\geq \frac{\ell_1\varepsilon}{2}>0,
\end{align*}
using (iii). Now, summing \eqref{eqStochApprox}, we get $X_n-X_0 \geq M_n+O_n$, where
\begin{equation}
M_n = \sum_{k=0}^{n-1}\gamma_k\xi_{k+1} \text{ and } O_n=\sum_{k=0}^{n-1}\gamma_k(A_1^{(k)}-K_1^{(k)}X_k+r_1^{(k)}). \notag
\end{equation}
Here, $O_n$ is $\mathcal{F}_{n-1}$-measurable and $M_n$ is a martingale satisfying
\[\E[M_{n+1}^2-M_n^2\mid\mathcal{F}_n]=\E[(M_{n+1}-M_n)^2\mid\mathcal{F}_n]=\E[\xi_{n+1}^2\gamma_n^2\mid\mathcal{F}_n]\leq C\gamma_n^2\]
by (i). Since $\gamma_n^2$ is summable by (ii), $M_n$ is $L^2$-bounded and hence there exists a finite $M_\infty$ such that $M_n\to M_{\infty}$ almost surely as $n\to\infty$. Moreover, by (ii), we have
\[R_n:= \sum_{k=0}^{n-1}\gamma_k r_1^{(k)}\to R_\infty <\infty,\]
almost surely as $n\to\infty$. Hence, for the $\varepsilon$ above, there exists $N_2$ such that
\begin{align*}
& \sup_{n\geq N_2}\sup_{m\geq 0}|M_{n+m}-M_n|\leq \frac{\varepsilon}{4} \ \text{ and } \ \sup_{n\geq N_2}\sup_{m\geq 0}|R_{n+m}-R_n|\leq \frac{\varepsilon}{4}.
\end{align*}
Now, fix some $n_0\geq N:=N_1\vee N_2$ for which $X_{n_0} < L_1-\varepsilon$. Let $\kappa_{n_0}:= \min\{t>n_0: X_t\geq L_1-\varepsilon\}$ be the first time after $n_0$ for which $X$ returns to $[L_1-\varepsilon,\infty)$. Then, for $m\geq 0$, using \eqref{eqStochApprox}, we have
\begin{align*}
X_{(n_0+m)\wedge \kappa_{n_0}}-X_{n_0} & \geq \big(M_{(n_0+m)\wedge\kappa_{n_0}}-M_{n_0}\big)+\big(R_{(n_0+m)\wedge\kappa_{n_0}}-R_{n_0}\big) \\
& \qquad +\sum_{k=n_0}^{((n_0+m)\wedge\kappa_{n_0}) -1}\gamma_k(A_1^{(k)}-K_1^{(k)}X_k) \\
&\geq -\tfrac{\varepsilon}{2}+\tfrac{\ell_1\varepsilon}{2}\sum_{k=n_0}^{((n_0+m)\wedge\kappa_{n_0})-1}\gamma_k.
\end{align*}
On $\{\kappa_{n_0}=\infty\}$, for $m\to\infty$, the left-hand-side remains finite since $X$ is a bounded process; however the right-hand-side diverges to infinity by (ii). Hence, $\kappa_{n_0}$ is almost surely finite. Furthermore, since $X_{(n_0+m)\wedge\kappa_{n_0}}\geq X_{n_0}-\varepsilon/2$ using the above calculation and $\gamma_k\geq 0$ for all $k$, the process $X$ returns to $[L_1-\varepsilon,\infty)$ without dropping below $X_{n_0}-\varepsilon/2$. Moreover, for large enough $n\geq N$,
\[X_{n+1}-X_n\geq -\tfrac{\varepsilon}{2}+\gamma_n(A_1^{(n)}-K_1^{(n)}X_n)\geq -\varepsilon,\]
as $A_1-K_1X$ is bounded and $\gamma_n$ tends to zero. Therefore, almost surely, $X_n\geq L_1-\varepsilon$ infinitely often, and for all except a finite number of $n$ any exit from $[L_1-\varepsilon,\infty)$ cannot drop under $L_1-2\varepsilon$; but starting from $[L_1-2\varepsilon,L_1-\varepsilon)$, the process $X$ returns to $[L_1-\varepsilon,\infty)$ before hitting $L_1-3\varepsilon$. Hence, $\liminf X_n\geq L_1-3\varepsilon$, almost surely. Since $\varepsilon$ was chosen arbitrarily, this concludes the proof.
\end{proof}
}

\subsection{Bounds on the empirical degree distribution}\label{SubSecBounds}
The aim of this section is to find bounds for the proportion of the vertices of degree at least $k$ located inside $[x_1,x_2]\subset (0,1)$. To this end, we define $P^{(n)}_{x_1,x_2}(k)$ as the proportion of vertices in $G_n$ that have degree at most $k$ and \edit{are} located inside the interval $[x_1,x_2]$, that is
\begin{equation*}\label{proportion}
P^{(n)}_{x_1,x_2}(k) = \frac{1}{n+n_0}\sum_{(v,x)\in G_n} \mathbbm{1}_{\set{\text{deg}_{G_n}(v)\leq k}}\mathbbm{1}_{\set{x\in[x_1,x_2]}}.
\end{equation*}
To get bounds on $P^{(n)}_{x_1,x_2}(k)$, we define the event that the new vertex $v_{n+1}$, arriving at time $n+1$, connects to a vertex of degree $k$ in $G_n$ which is located inside $[x_1,x_2]$. We denote this event by $E_{n+1}$. We cannot give a precise description of the probability of $E_{n+1}$, however we can bound it from above and below in a natural way. The estimations are made in the part in which one of the $r$ candidates of the sample is chosen for connection. \edit{First, we consider the function
\begin{equation}\label{deff1}
f_1(y_1,y_2):=\sum_{s=1}^r \Xi_s \left(\sum_{j=0}^{s-1}\sum_{i=s}^{r}{r\choose i}{i\choose j}y_1^j(y_2-y_1)^{i-j-1}(1-y_2)^{r-i}\right).
\end{equation}
Upon multiplying $f_1$ by $(y_2-y_1)$, the term inside the outer brackets states that for $r$ points ranked by $1,\dots,r$, the first $j$ points are sampled from the interval $(0,y_1)$, the next $i-j$ points (including point $s$) are sampled from $[y_1,y_2]$ and the remaining $i$ points are from $(y_2,1)$. Since we do not consider the precise ordering within the intervals, an upper bound for the probability of choosing a vertex for connection from the candidate sample with location inside $[x_1,x_2]$ is given by 
\[f_1(\Psi_n(x_1),\Psi_n(x_2))(\Psi_n(x_2) - \Psi_n(x_1)).\] 
Secondly, we denote $\mathbb{P}_n$ as the conditional probability measure given by the graph $G_n$ and all locations of vertices contained within. We also denote $w_n$ as the vertex which has been chosen for connection at time $n+1$. We have
\begin{align*}
\mathbb{P}_n &\big(w_n \text{ has degree }k \ \big| \ w_n\text{ is located inside } [x_1,x_2]\big) \\
& = \frac{\mathbb{P}_n\big(w_n \text{ has degree }k \text{ and is located inside }[x_1,x_2]\big)}{\Psi_n(x_2) - \Psi_n(x_1)} \\
& = \tfrac{k+\alpha}{(2+\alpha)(n+n_0)-2} (n+n_0)\left(P^{(n)}_{x_1,x_2}(k)-P^{(n)}_{x_1,x_2}(k-1)\right)\tfrac{1}{\edit{\Psi_n(x_2) - \Psi_n(x_1)}}.
\end{align*} 
This holds because the first factor is the probability in which a vertex of degree $k$ is sampled according to \eqref{PA} and $(n+n_0)\big(P^{(n)}_{x_1,x_2}(k)-P^{(n)}_{x_1,x_2}(k-1)\big)$ counts the number of degree $k$ vertices with corresponding locations inside the interval $[x_1,x_2]$ in $G_n$. Therefore, both parts together yield
\begin{align}
\mathbb{P}_n(E_{n+1}) \leq \tfrac{(k+\alpha)(n+n_0)}{(2+\alpha)(n+n_0)-2}\left(P^{(n)}_{x_1,x_2}(k)-P^{(n)}_{x_1,x_2}(k-1)\right) f_1(\Psi_n(x_1),\Psi_n(x_2)). \label{eqfUpperbound}
\end{align}
}

\edit{Similarly, we can achieve a lower bound for $\mathbb{P}_n(E_{n+1})$ if we only consider samples of candidates where exactly one vertex is located inside $[x_1,x_2]$.} Thus, we obtain 
\begin{align}
\edit{\mathbb{P}_n}(E_{n+1}) \geq \frac{(k+\alpha)(n+n_0)}{(2+\alpha)(n+n_0)-2}\left(P^{(n)}_{x_1,x_2}(k)-P^{(n)}_{x_1,x_2}(k-1)\right) f_2(\Psi_n(x_1),\Psi_n(x_2)), \label{eqfLowerbound}
\end{align}
where $f_2$ is given by
\begin{equation}\label{deff2}
f_2(y_1,y_2) = \sum_{s=1}^r s\Xi_s {r\choose s}y_1^{s-1}(1-y_2)^{r-s}.
\end{equation}
With these bounds, which will be crucial for the asymptotic degree later, we are ready to proceed to the stochastic approximation.
\begin{lem}\label{ex num vert in G deg leq k}
Let $k\in \mathbb{N}$ and $0< x_1 < x_2< 1$. Define for $j\in \set{1,2}$ the random variables
\begin{equation*}
A^{(n)}_j = A^{(n)}_j(k) := \frac{k+\alpha}{2+\alpha} f_j\left(\Psi_n(x_1), \Psi_n(x_2)\right) P^{(n)}_{x_1,x_2}(k-1) + (x_2 - x_1)
\end{equation*}
and
\begin{equation*}
K^{(n)}_j = K^{(n)}_j(k) := 1 + \frac{k+\alpha}{2+\alpha} f_j\left(\Psi_n(x_1), \Psi_n(x_2)\right),
\end{equation*}
for $n\in \mathbb{N}_0$. Let $\mathcal{F}_{n}$ be the filtration generated by the sequence of graphs $(G_i,x_{i};i\leq n)$ and, for $n\in \mathbb{N}_0$,
\begin{equation*}
\xi^{(n+1)} = (n + n_0 + 1)\left( P^{(n+1)}_{x_1,x_2}(k) - \mathbb{E}\left[P^{(n+1)}_{x_1,x_2}(k) \big| \mathcal{F}_n\right]\right).
\end{equation*}
Then, for the growth of the proportion of vertices with degree at most $k$ and location inside $[x_1,x_2]$, it holds 
\begin{align*}\label{sae inequ}
\frac{A^{(n)}_1 - K^{(n)}_1 P^{(n)}_{x_1,x_2}(k) + \xi^{(n+1)}-R^{(n)}}{n+n_0+1} & \leq P^{(n+1)}_{x_1,x_2}(k)-P^{(n)}_{x_1,x_2}(k) \\ & \leq\frac{A^{(n)}_2 - K^{(n)}_2 P^{(n)}_{x_1,x_2}(k) + \xi^{(n+1)}}{n+n_0+1},
\end{align*}
where \edit{$R^{(n)}$ is a non-random error term satisfying} $R^{(n)}/(n+n_0+1)=O(n^{-2})$ as $n\to\infty$.
\end{lem}
\begin{proof}
Since
\begin{equation*}
P^{(n+1)}_{x_1,x_2}(k)-P^{(n)}_{x_1,x_2}(k) = \mathbb{E}\left(P^{(n+1)}_{x_{1},x_{2}}(k)\Big|\mathcal{F}_{n}\right) + \frac{\xi^{(n+1)}}{n + n_0 + 1} -P^{(n)}_{x_1,x_2}(k),
\end{equation*}
it is sufficient to find bounds for the expected increase in the number of vertices with degree at most $k$ and location inside the interval $[x_1,x_2]$ when $v_{n+1}$ joins the graph with location $x_{n+1}$, given $G_n$. This can be expressed by 
\begin{equation*}\label{expected increase}
\mathbb{E}\left((n+n_0+1)P^{(n+1)}_{x_{1},x_{2}}(k)\Big|\mathcal{F}_{n}\right)=(n+n_0)P^{(n)}_{x_{1},x_{2}}(k)+\mathbb{P}(x_{n+1}\in [x_1,x_2]) - \edit{\mathbb{P}_{n}(E_{n+1})},
\end{equation*}
The first term here counts the number of degree at most $k$ vertices in $G_n$ with locations in the interval $[x_1,x_2]$. The second term is the probability that the location of the new vertex $v_{n+1}$ falls into the same interval\edit{. Both $\mathbb{P}_n$ and $E_{n+1}$ are as defined as above. We have utilized here the fact that $P_{x_1,x_2}^{(n)}(k)$ is $\mathcal{F}_n$-measurable, the new location $x_{n+1}$ is independent of $\mathcal{F}_n$, and that the event $E_{n+1}$ only depends on the graph $G_n$ and the corresponding locations of vertices contained within.}
As the locations are i.i.d. uniform, this probability is equal to $x_2-x_1$. For the probability of the event $E_{n+1}$ an upper and lower bound is given by \eqref{eqfUpperbound} and \eqref{eqfLowerbound}. Hence, we have
\begin{equation*}
\begin{split}
\mathbb{E}\left((n+n_0+1)P^{(n+1)}_{x_{1},x_{2}}(k)\Big|\mathcal{F}_{n}\right)\leq &\frac{k+\alpha}{2+\alpha} f_2(\Psi_n(x_{1}),\Psi_n(x_{2}))\left(P^{(n)}_{x_{1},x_{2}}(k-1)-P^{(n)}_{x_{1},x_{2}}(k)\right)\\
&+(n+n_0)P^{(n)}_{x_{1},x_{2}}(k)+(x_{2}-x_{1})
\end{split}
\end{equation*}
as well as
\begin{equation*}
\begin{split}
\mathbb{E}\left((n+n_0+1)P^{(n+1)}_{x_{1},x_{2}}(k)\Big|\mathcal{F}_{n}\right)\geq &\frac{k+\alpha}{2+\alpha} f_1(\Psi_n(x_{1}),\Psi_n(x_{2}))\left(P^{(n)}_{x_{1},x_{2}}(k-1)-P^{(n)}_{x_{1},x_{2}}(k)\right)\\
&+(n+n_0)P^{(n)}_{x_{1},x_{2}}(k)+(x_{2}-x_{1}) - R^{(n)},
\end{split}
\end{equation*}
where $R^{(n)}=O(n^{-1})$ is an error term, occurring as the difference of the given bound in \eqref{eqfUpperbound} and the first summand on the right-hand side of the equation together with the fact that $(P^{(n)}_{x_1,x_2}(k-1)-P^{(n)}_{x_1,x_2}(k))\geq -1$ and the boundedness of $f_1$. 
\end{proof}
Since the number of vertices with degree at most $k$ and location inside $[x_1,x_2]$ can change by at most one if we add a new vertex $v_{n+1}$ to the graph $G_n$, the noise $\xi^{(n)}$ defined in Lemma \ref{ex num vert in G deg leq k} is absolutely bounded by one. Additionally, it holds $\mathbb{E}\left[ \xi^{(n+1)} \vert \mathcal{F}_n \right] = 0$ by its definition. Therefore, we can use stochastic approximation techniques to construct bounds for the asymptotic behaviour of the proportion of vertices with degree at most $k$ and location inside $[x_1,x_2]$.
\begin{lem}\label{limbound}
For $\edit{\alpha\geq \alpha_c}$ and all $k\in \mathbb{N}$, the proportion of vertices with degree at most $k$ and location inside $[x_{1},x_{2}]\subset(0,1)$ satisfies
\[
\edit{L_1(k) \leq \liminf_{n\to \infty} P_{x_1,x_2}^{(n)}(k) \leq \limsup_{n\to\infty}P_{x_1,x_2}^{(n)}(k)\leq L_2(k)}
\]
almost surely, where
\begin{align*}
L_j(k) = (x_2-x_1)\left(1-\frac{\Gamma\left(\alpha+1+\frac{2+\alpha}{f_j(\Psi(x_1),\Psi(x_2))}\right)\Gamma(\alpha+1+k)}{\Gamma(\alpha+1)\Gamma\left(\alpha+1+k+\frac{2+\alpha}{f_j(\Psi(x_1),\Psi(x_2))}\right)}\right), \ \text{ for }j\in \set{1,2}.
\end{align*}
\end{lem}
\begin{proof}
We prove the result by applying Lemma~\ref{conflem} to the observed bounds in Lemma~\ref{ex num vert in G deg leq k}. Here, we focus on the lower bound \edit{for the $\liminf$}; the upper bound \edit{ for the $\limsup$} follows by replacing $f_1$ by $f_2$ in the arguments. \edit{For each fixed $k$, the boundedness conditions on $A_1^{(n)}$ and $K_1^{(n)}$, defined in Lemma~\ref{ex num vert in G deg leq k}, and $P_{x_1,x_2}^{(n)}(k)$ as well as the assumptions (ii) and (iii) are straight forward to check. We have already shown (i) above.} We must now show that $A_1^{(n)}$ and $K_1^{(n)}$ converge such that $\edit{\liminf_{n\to\infty}}(A_1^{(n)}/K_1^{(n)})\geq L_1(k)$, for every $k$. First note that, since $\edit{\alpha \geq \alpha_c}$, $\Psi_n(x)$ converges almost surely to $\Psi(x)$ as defined in Section~\ref{SubSecPhase} for all $x\in [0,1]$. Hence, $K^{(n)}_1$ converges almost surely to $1 + \frac{k+\alpha}{2+\alpha} f_1\left(\Psi(x_1), \Psi(x_2)\right)$.
Now, the theorem can be derived by induction. Let $k=1$, then by definition $A_1^{(n)} = x_2 - x_1$. Hence, $\edit{\liminf}_{n\to \infty} P_{x_1,x_2}^{(n)}(1) \geq L_1(1)$ almost surely. Assume that for an arbitrary fixed $k\in \mathbb{N}$ the stated lower bound holds. Then, for $k+1$, we get
\[
\edit{\liminf_{n\to \infty}} A^{(n)}_1 \geq \frac{k+1+\alpha}{2+\alpha} f_1(\Psi(x_1), \Psi(x_2))L_1(k) + (x_2 - x_1)
\]
almost surely and hence
\begin{align*}
\edit{\liminf\limits_{n\to \infty}} P_{x_1,x_2}^{(n)}(k+1) &\geq \frac{\frac{k + 1 +\alpha}{2+\alpha} f_1(\Psi(x_1), \Psi(x_2))L_1(k) + (x_2 - x_1)}{1 + \frac{k+1+\alpha}{2+\alpha} f_1\left(\Psi(x_1), \Psi(x_2)\right)}\\
&=(x_2 - x_1)\left(1 - \frac{\Gamma\left(\alpha+1+\frac{2+\alpha}{f_j(\Psi(x_1),\Psi(x_2))}\right)\Gamma(\alpha+2+k)}{\Gamma(\alpha+1)\Gamma\left(\alpha+2+k+\frac{2+\alpha}{f_j(\Psi(x_1),\Psi(x_2))}\right)}\right)
\end{align*}
almost surely. 
\end{proof}
\subsection{Limiting degree distribution}\label{SubSecLimitGed}
In this section, we use the established bounds of Lemma~\ref{limbound} to prove the main results stated in Section~\ref{SubSecMainResult}. To this end, we consider now the proportion of vertices located within some interval that have a given maximum degree. We show, that in the late time regime this proportion converges, by shrinking the interval to a single point, to some probability kernel $\mu$ on $\mathcal{P}(\N)\times (0,1)$. Here, $\mathcal{P}(\N)$ denotes the set of all subsets of $\N$. We show that this probability kernel $\mu$ is heavy tailed, proving Theorem~\ref{ThmLocalDegree} as $\nu(k,x)=\mu(\{k,k+1,\dots\},x)$

\begin{lem}\label{LemLocalDegree}
If $\alpha\geq \alpha_c$, there exists a probability kernel $\mu:\mathcal{P}(\N)\times(0,1)\to [0,1]$ such that
\begin{enumerate}
\item[(i)] Almost surely, 
\[\lim_{x_1\downarrow x}\lim_{n\rightarrow\infty} \frac{P^{(n)}_{x,x_1}(k)}{x_1-x}=\mu(\{1,\dots,k\},x).\]
\item[(ii)] $\mu(\{k,k+1,\dots\},x) = k^{-\frac{2+\alpha}{f(\Psi(x))}+o(1)}$, as $k\rightarrow\infty$.
\end{enumerate}
\end{lem}
\begin{proof}
Note that the functions $f_1$ and $f_2$, defined in \eqref{deff1} and \eqref{deff2}, both converge to the same limit as $y_1\downarrow y$, namely
\[\lim_{y_1\downarrow y}f_1(y,y_1)=\lim_{y_1\downarrow y}f_2(y,y_1)=\sum_{s=1}^r s\Xi_s {r\choose s}y^{s-1}(1-y)^{r-s},\]
that is $f(y)$, the function used to describe the condensation phase transition in Section \ref{SubSecPhase}. Sending first $n\to\infty$ and applying then the limit $x_1\to x$ on the bounds observed in Lemma~\ref{limbound}, we get by continuity of $f$ and $\Psi$,
\begin{equation} \label{kernelCDF}
\lim_{x_1\downarrow x}\lim_{n\rightarrow\infty} \frac{P^{(n)}_{x,x_1}(k)}{x_1-x} = 1-\frac{\Gamma\left(\alpha+1+\frac{2+\alpha}{f(\Psi(x))}\right)}{\Gamma(\alpha +1)}\frac{\Gamma(\alpha+1+k)}{\Gamma\left(\alpha+1+k+\frac{2+\alpha}{f(\Psi(x))}\right)}.
\end{equation}
Note that
\begin{equation}\label{Stirling}
\frac{\Gamma(\alpha+1+k)}{\Gamma\left(\alpha+1+k+\frac{2+\alpha}{f(\Psi(x))}\right)} \sim (\alpha+1+k)^{-\frac{2+\alpha}{f(\Psi(x))}}\edit{\sim k^{-\frac{2+\alpha}{f(\Phi(x))}}}, \text{ as } k\uparrow\infty
\end{equation}
by Stirling's formula. For fixed $x\in(0,1)$, the right-hand side of \eqref{kernelCDF} converges to one as $k\rightarrow\infty$, and hence defines a distribution function. Moreover, for fixed $k$, the right-hand side of \eqref{kernelCDF} is continuous in $x$. Therefore, the desired probability kernel $\mu$ exists, proving (i). The tail behaviour stated in (ii) is an immediate consequence of \eqref{Stirling}.
\end{proof}
\edit{Since the empirical distribution of the vertices' locations converges to the uniform distribution on $(0,1)$}, we can now use the probability kernel to properly describe $\mu_k$, the asymptotic proportion of vertices with degree at least $k$ by integrating $\mu(\{k,k+1,\dots\},x)$ with respect to the location. Namely,
\[\mu_k = \int_0^1 \mu(\{k,k+1,\dots\}, x)dx.\]
We are now ready to prove Theorem~\ref{ThmPowerLaw}.

\begin{proof}[Proof of Theorem \ref{ThmPowerLaw}]
Since we only consider the case when $k$ is (very) large, we want to apply a saddle point method approach. To this end, write 
\[g(x)=\frac{\Gamma\left(\alpha+1+\frac{2+\alpha}{f(x)}\right)}{\Gamma(\alpha +1)}\]
and consider
\[\mu_k = \int_0^1 \mu(\{k,k+1,\dots\},x)dx = \edit{\int_0^1 g(\Psi(x))\tfrac{\Gamma(\alpha+1+k)}{\Gamma\big(\alpha+1+k+\tfrac{2+\alpha}{f(\Phi(x)}\big)}dx}\]
given from the proof of Lemma~\ref{LemLocalDegree}. \edit{Using \eqref{Stirling} this reads
\[\mu_k \sim\int_0^1 g(\Phi(x))\exp\big(-\tfrac{2+\alpha}{f(\Phi(x))}\log(k)\big)dx,\]
as $k\to\infty$. If $f\equiv c$ is constant (e.g.\ when $\Xi_s=1/r$), the exponent does not depend on $x$ and the claim follows immediately. Hence, we assume that $f$ is non constant.} Now, since we work in the noncondensation phase, $\Psi(x)$ is the unique zero of $F_1(y;x,\Xi)$, defined in Section~\ref{SubSecPhase}. Due to the structure of $F_1(\Psi(x); x,\Xi)=0$, we can see that the inverse of $\Psi$ exists and that it is a polynomial. Thus, it is differentiable. Together with $\Psi(0)=0$, almost surely, and $\Psi(1)=1$ a change of variable leads to
\[\mu_k\edit{\sim}\int_0^1 (\Psi^{-1})'(y)\cdot g(y)\exp\left(-\frac{2+\alpha}{f(y)}{\log(k)}\right)dy\]
\edit{as $k\to\infty$}. For $k\to\infty$, this integral gets dominated by its largest peak that is located at the minimum value of $(2+\alpha)/f(y)$, occurring at the maximum value of $f(y)$. Since $f$ is a non-negative polynomial, there exists some $x_0\in [0,1]$ that maximizes $f$. \edit{In the case that $x_0$ is not uniquely determined, we can split $[0,1]$ in finitely many disjoint subintervals such that each subinterval only contains exactly one maximizer. We then integrate these subintervals separately which leads to a sum of integrals all of the same order.}
Moreover, we know that the second derivative of $f$ exists and that $-\left((2+\alpha)/f\right)''(x_0)>0$ as well as $(\Psi^{-1})'(x_0)\cdot g(x_0)>0$. Hence, we get by the saddle point method, for some constant $C$ and with $2+\alpha_c=f(x_0)$ that
\begin{equation*}
\mu_k \edit{\sim} C\sqrt{\frac{2 \pi}{\log(k+\alpha)}}\exp\left(-\frac{2+\alpha}{2+\alpha_c}\log(k+\alpha)\right)\cdot\left(1+O\left(\frac{1}{\log(k+\alpha))}\right)\right), \label{saddlePoint}
\end{equation*}
as $k\to \infty$, which yields the desired result.
\end{proof}
\section{Examples and simulations}\label{Section 4}
In this section, we discuss a number of examples of the model and use the stated results to calculate the critical value $\alpha_c$ and the power law exponent $\tau$. Simulations of the model back up those results and showcase the different behaviour of the local degree distribution. For this, the different examples are simulated for an initial tree graph of $100$ vertices where $1000000$ new vertices are added to the graph. The code for the simulations can be freely accessed at: \url{http://www.mi.uni-koeln.de/~agrauer/files/code/PA_with_location.R}

\begin{figure}[t]
\centering
\begin{subfigure}[b]{0.48\textwidth}\centering 
\includegraphics[width=\textwidth]{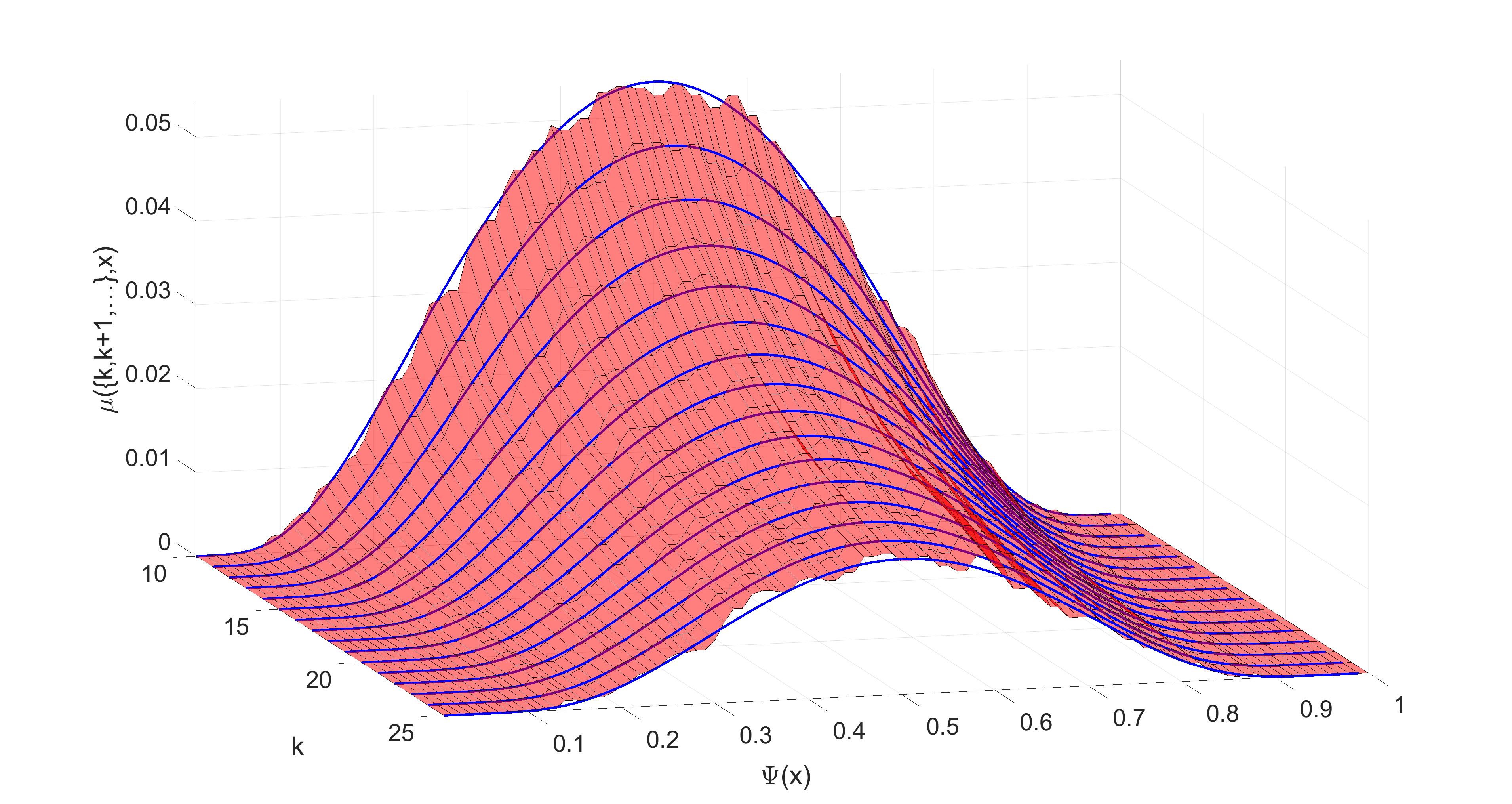}
\caption[]{{$\Xi = (0,1,0)$.}}\label{fig:(1)}
\end{subfigure}
\begin{subfigure}[b]{0.48\textwidth}\centering 
\includegraphics[width=\textwidth]{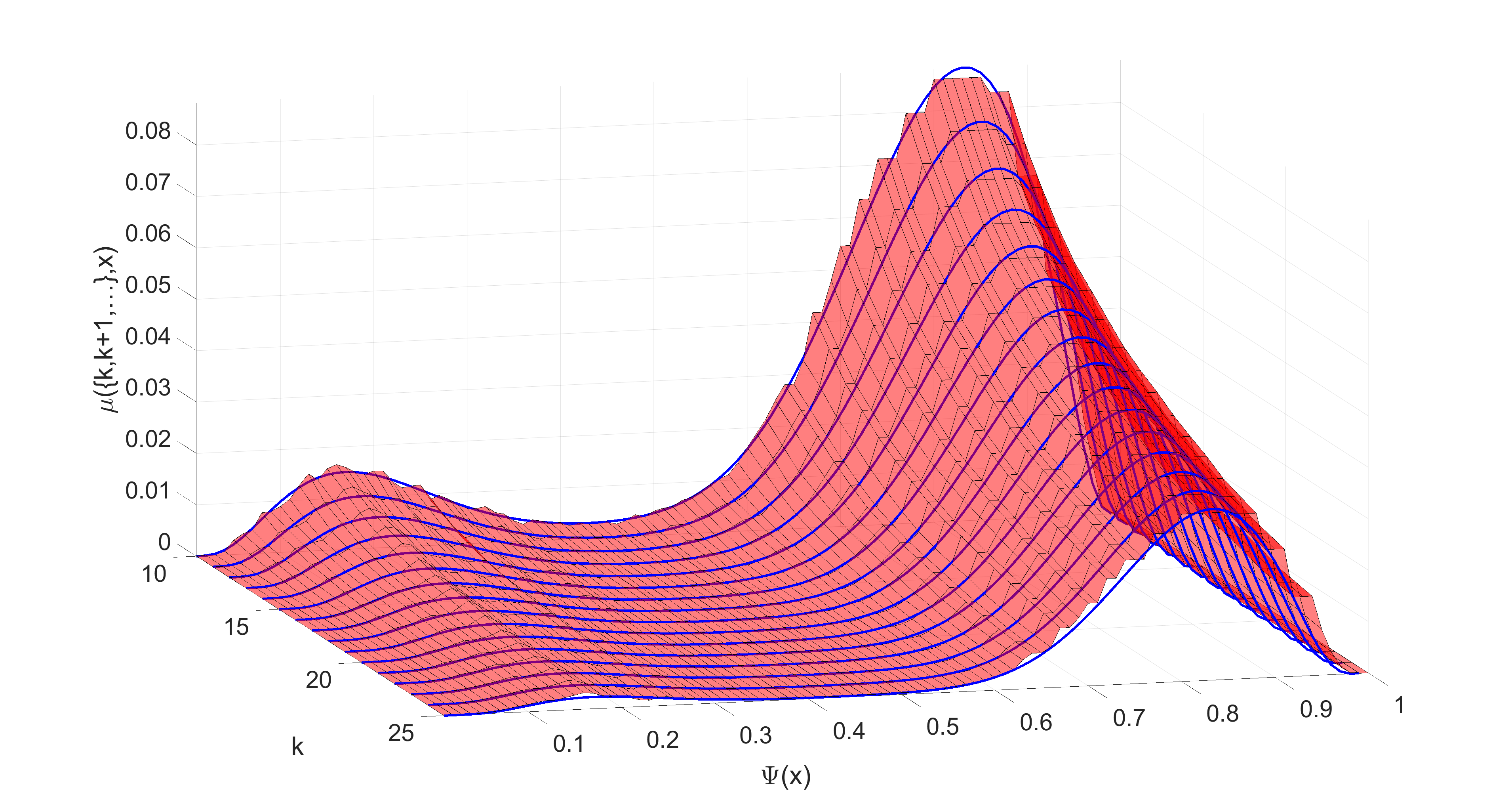}
\caption[]{{$\Xi = \left(0,\frac{1}{2},0,0,0,\frac{1}{2},0\right)$.}}\label{fig:(2)}
\end{subfigure}
\begin{subfigure}[b]{0.48\textwidth}\centering 
\includegraphics[width=\textwidth]{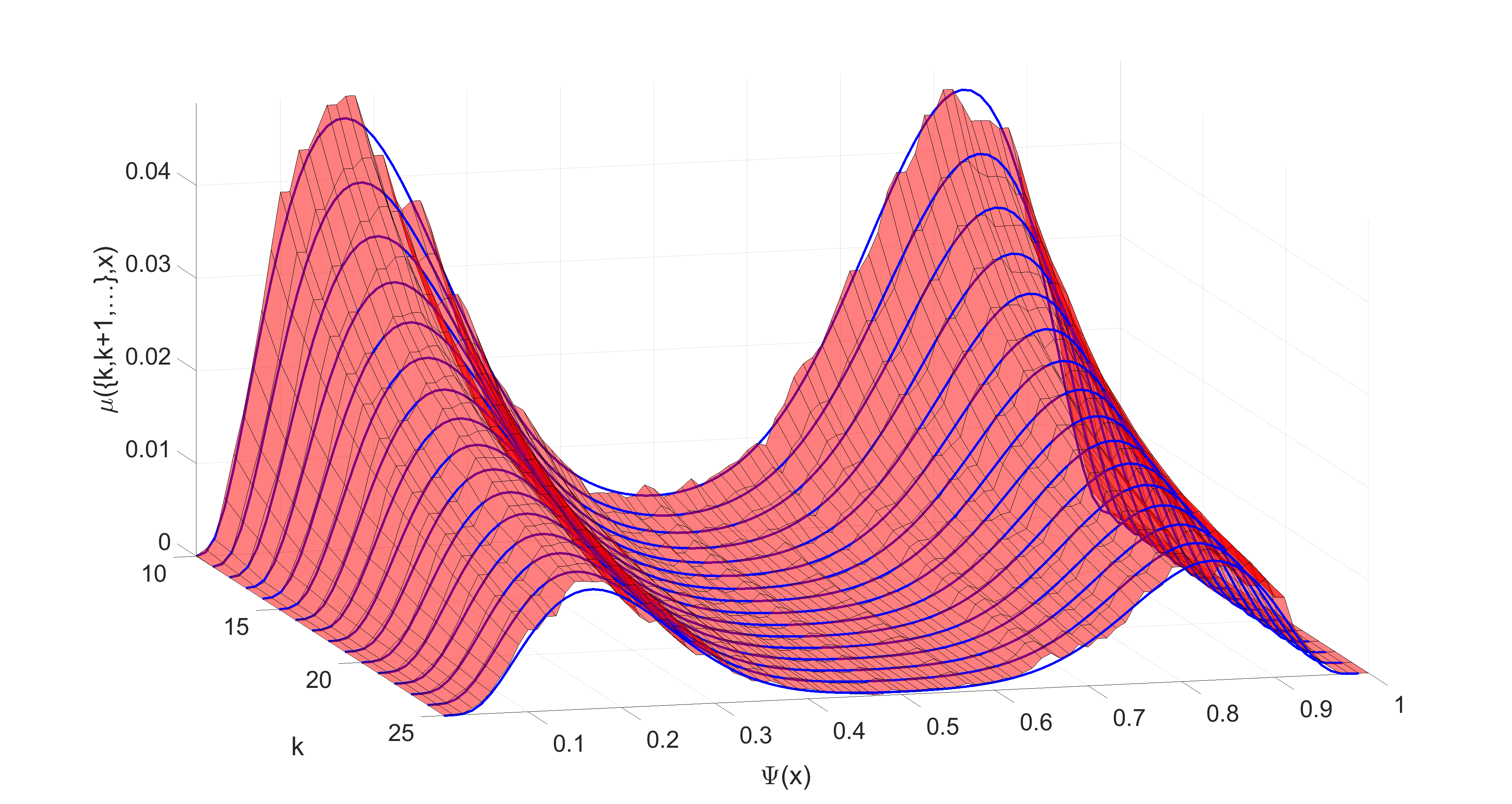}
\caption[]{{$\Xi = \left(0,\frac{1}{3},0,0,0,\frac{2}{3},0\right)$.}}\label{fig:(3)}
\end{subfigure}
\begin{subfigure}[b]{0.48\textwidth}\centering 
\includegraphics[width=\textwidth]{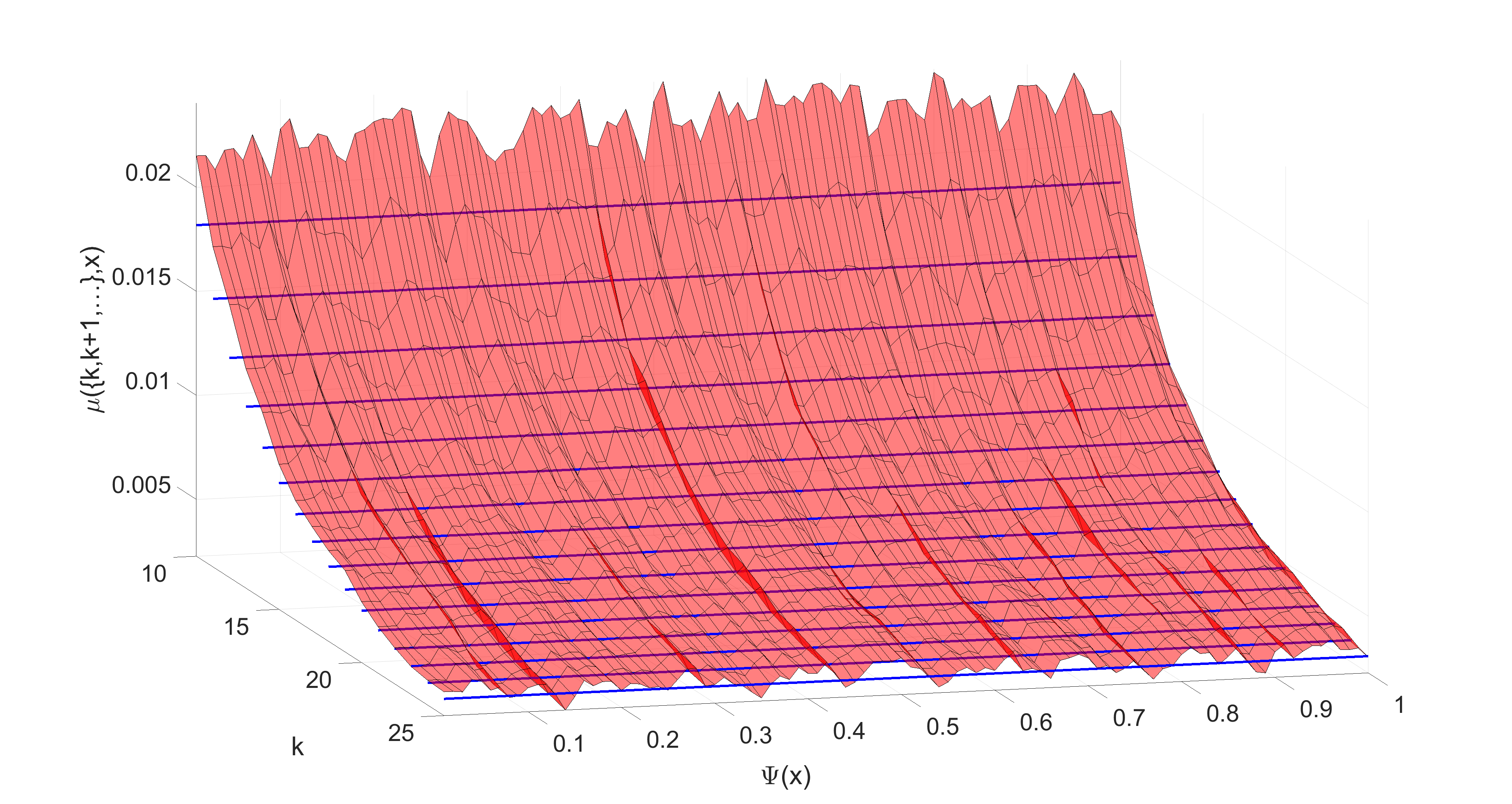}
\caption[]{{$\Xi = \left(\frac{1}{7},\frac{1}{7},\frac{1}{7},\frac{1}{7},\frac{1}{7},\frac{1}{7},\frac{1}{7}\right)$.}}\label{fig:(4)}
\end{subfigure}
\caption[]{Simulation of the local degree distribution for the \edit{three examples of this section. We have inserted picture (d), which coincides to standard preferential attachment, for comparison.} The red surface shows the simulation results while the blue curves depicts the analytical result of Lemma \ref{LemLocalDegree} for each $k$. Each plot is generated for $\Psi(x)\in(0,1)$ and $k\in[10,25]$ and $\alpha = 0$.} \label{fig:simulations}
\end{figure}
Throughout this section, we denote by $f_\text{max}$ the global maximum of $f$ on $[0,1]$. The first example is the \emph{middle of three model} introduced in \cite{Sou.2}. This model corresponds to the selection vector $\Xi=(0,1,0)$, which implies $f(y)=6y(1-y)$ due to equation \eqref{connectprobwoPA}. This function is maximized at $y = 1/2$ giving $f_\text{max} = 3/2$. As seen in Figure~\ref{fig:(1)}, $y$ coincides with the maximiser of the local degree distribution $\mu([k,\infty),x)$, for any $k\in \mathbb{N}$. Using the method introduced in Section~\ref{SubSecPhase}, the critical value is $\alpha_c=-1/2$, agreeing with the results in \cite{Sou.2}. By Theorem ~\ref{ThmPowerLaw} it can be seen that the degree distribution associated with the middle of three model follows
\begin{equation*}
\mu_k=k^{-\frac{2+\alpha}{3/2}+o(1)},\text{ as }k\uparrow\infty.
\end{equation*}
Introduced in \cite{Sou.2} is the \emph{second or sixth of seven model}, corresponding to $\Xi=(0,\frac{1}{2},0,0,0,\frac{1}{2},0)$. Hence, the associated function is $f(y)=21y(1-y)\left((1-y)^4+y^4\right)$. This leads to $f_\text{max}=\frac{7\left(5\sqrt{10}-14\right)}{9}$. Unlike the middle of three model, $f$ has two maximisers which are also peaks of the local degree distribution, see Figure \ref{fig:(2)}. The critical value for this example is $\alpha_c=\frac{35\sqrt{10}-116}{9}\approx-0.591$ and it holds
\begin{equation*}
\mu_k=k^{-\frac{9(2+\alpha)}{7\left(5\sqrt{10}-14\right)}+o(1)},\text{ as }k\uparrow\infty.
\end{equation*}
The final example is an \emph{asymmetric} version of the second or sixth of seven model, i.e. $\Xi=(0,\frac{1}{3},0,0,0,\frac{2}{3},0)$ as selection vector leading to $f(y)=14y(1-y)\left((1-y)^4+2y^4\right)$. Although this function has two local maximisers, we only care about the global maximum point with $f_\text{max}\approx1.8769$. Figure~\ref{fig:(3)} shows that the mass of the local degree distribution vanishes for large $k$ at the non-global maximiser but concentrates at the global one. The estimation of $f_\text{max}$ leads to the critical value $\alpha_c\approx -0.1231$ and 
\begin{equation*}
\mu_k\approx k^{-\frac{2+\alpha}{1.8769}+o(1)},\text{ as }k\uparrow\infty.
\end{equation*}
Although the proof of Theorem~\ref{ThmPowerLaw} only shows slow convergence to the stated result, our simulations show the stated power law behaviour. For the following figure the simulated degree distribution of the models is fitted to $k^{-\tau}$, considering the logarithmic correction term arising in the proof of Theorem~\ref{ThmPowerLaw}. For large $\alpha$, it is necessary to consider simulations of bigger graphs, since the degree is less important for the preferential attachment mechanism, which leads to a small maximum degree of the model. Note that in Figure \ref{fig:exponent} the power law exponent of the simulations in each example converges to $1$ as $\alpha \to \alpha_c$.
\begin{figure}[H]
\centering
\begin{subfigure}[b]{0.32\textwidth}\centering 
\includegraphics[width=\textwidth]{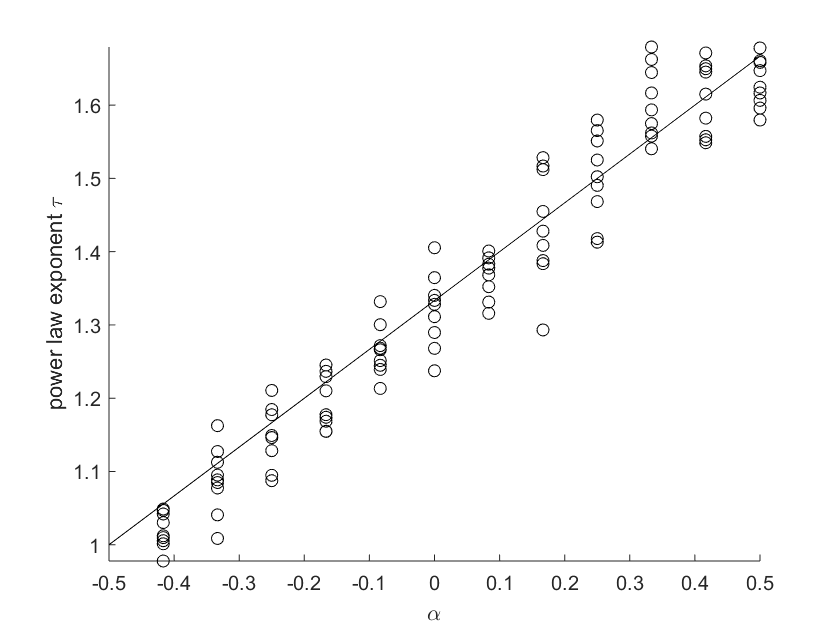}
\caption[]{{$\Xi = (0,1,0)$.}}\label{fig:(1e)}
\end{subfigure}
\begin{subfigure}[b]{0.32\textwidth}\centering 
\includegraphics[width=\textwidth]{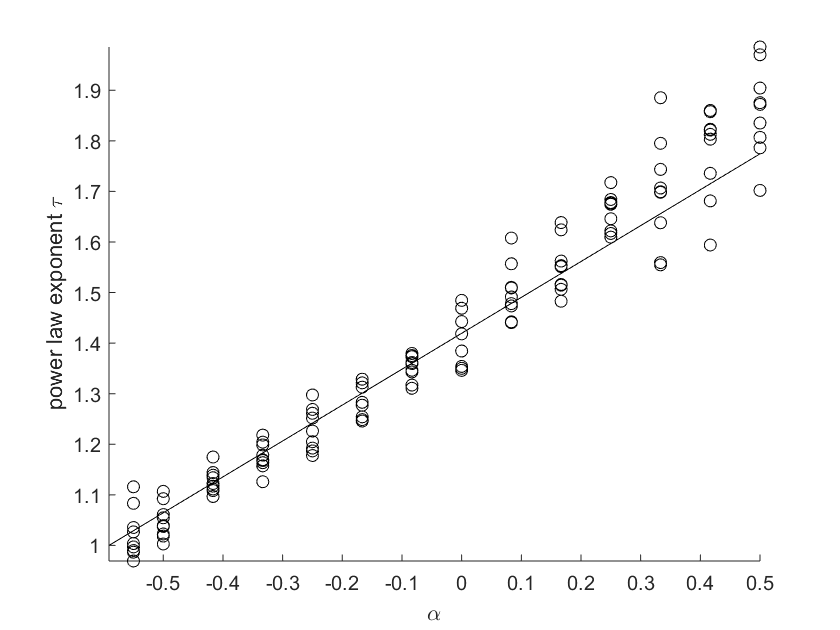}
\caption[]{{$\Xi = \left(0,\frac{1}{2},0,0,0,\frac{1}{2},0\right)$.}}\label{fig:(2e)}
\end{subfigure}
\begin{subfigure}[b]{0.32\textwidth}\centering 
\includegraphics[width=\textwidth]{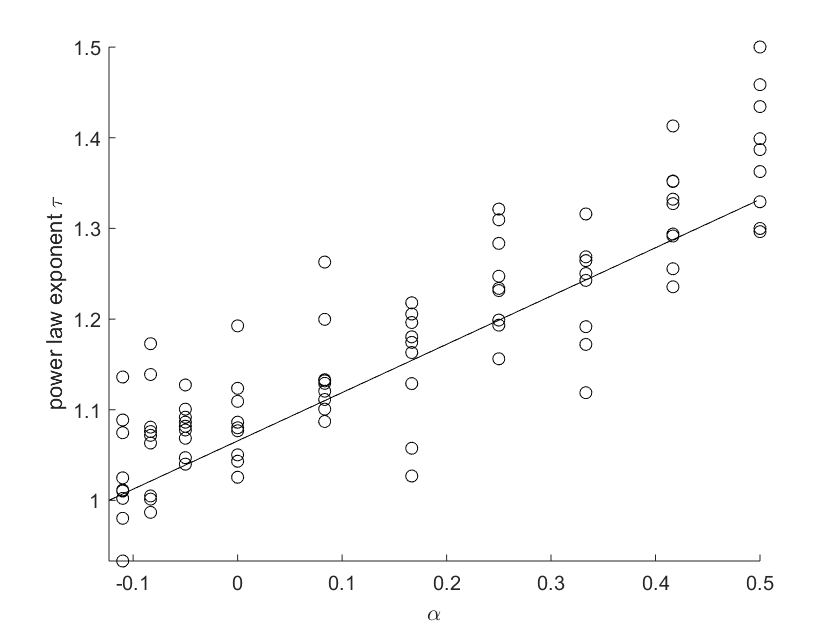}
\caption[]{{$\Xi = \left(0,\frac{1}{3},0,0,0,\frac{2}{3},0\right)$.}}\label{fig:(3e)}
\end{subfigure}
\caption[]{Simulations of the power law exponent of the degree distribution for each example for $\alpha$ between the corresponding $\alpha_c$ and $1/2$. The lines show the analytical result of Theorem~\ref{ThmPowerLaw}.} \label{fig:exponent}
\end{figure}
\section*{Acknowledgements}
The authors would like to thank the \emph{Heilbronn Institute for Mathematical Research} for their generous support in the form of a grant awarded under their `2019 Heilbronn Focused Research Grants' initiative.

\edit{The authors would also like to thank Jonathan Jordan for his guidance and direction throughout the writing of this article and Peter Gracar for his patient proof reading and generous support.}

\edit{Finally, we appreciate the effort which the associate editor and both referees have put into this article. Their detailed comments and constructive feedback were particularly useful.}


\begin{thebibliography}{19}
\bibliographystyle{abbrv}
\footnotesize 
\bibitem{P.10} A.-L. Barab{\'a}si and R.~Albert. \newblock Emergence of scaling in random networks. \newblock {\em Science}, 286(5439):509--512, 1999.

\bibitem{P.1} G.~Bianconi and A.-L. Barab{\'a}si. \newblock Competition and multiscaling in evolving networks. \newblock {\em EPL (Europhysics Letters)}, 54(4):436, 2001.

\bibitem{P.36} B.~Bollob\'as, O.~Riordan, J.~Spencer, G.~Tusn\'ady, et~al. \newblock The degree sequence of a scale-free random graph process. \newblock {\em Random Structures \& Algorithms}, 18(3):279--290, 2001.

\bibitem{P.2} C.~Borgs, J.~Chayes, C.~Daskalakis, and S.~Roch. \newblock First to market is not everything: an analysis of preferential attachment with fitness. \newblock In {\em Proceedings of the thirty-ninth annual ACM symposium on theory of computing}, pages 135--144. ACM, 2007.

\bibitem{P.3} S.~Dereich, C.~Mailler, P.~M{\"o}rters, et~al. \newblock Nonextensive condensation in reinforced branching processes. \newblock {\em The Annals of Applied Probability}, 27(4):2539--2568, 2017.

\bibitem{P.15} S.~Dereich and P.~M{\"o}rters. \newblock Emergence of condensation in Kingman's model of selection and
mutation. \newblock {\em Acta Applicandae Mathematicae}, 127(1):17--26, 2013.

\bibitem{P.4} S.~Dereich and M.~Ortgiese. \newblock Robust analysis of preferential attachment models with fitness. \newblock {\em Combinatorics, Probability and Computing}, 23(3):386--411, 2014.

\bibitem{P.42} S.~N. Dorogovtsev, J.~F.~F. Mendes, and A.~N. Samukhin. \newblock Structure of growing networks with preferential linking. \newblock {\em Physical Review Letters}, 85(21):4633, 2000.

\bibitem{Sou.1} N.~Freeman and J.~Jordan. \newblock Extensive condensation in a model of preferential attachment with fitnesses. \newblock {\em Electronic Journal of Probability}, 25, 2020.

\bibitem{P.5} J.~Haslegrave and J.~Jordan. \newblock Preferential attachment with choice. \newblock {\em Random Structures \& Algorithms}, 48(4):751--766, 2016.

\bibitem{Sou.2} J.~Haslegrave, J.~Jordan, and M.~Yarrow. \newblock Condensation in preferential attachment models with location-based choice. \newblock {\em Random Structures and Algorithms} 56(3):775--795, 2020.

\bibitem{P.24} J.~Jordan. \newblock The degree sequences and spectra of scale-free random graphs. \newblock {\em Random Structures \& Algorithms}, 29(2):226--242, 2006.

\bibitem{P.11} J.~Jordan. \newblock Geometric preferential attachment in non-uniform metric spaces. \newblock {\em Electronic Journal of Probability}, 18, 2013.

\bibitem{P.12} J.~Jordan and A.~R. Wade. \newblock Phase transitions for random geometric preferential attachment graphs. \newblock {\em Advances in Applied Probability}, 47(2):565--588, 2015.

\bibitem{P.6} P.~Krapivsky and S.~Redner. \newblock Choice-driven phase transition in complex networks. \newblock {\em Journal of Statistical Mechanics: Theory and Experiment}, 2014(4):P04021, 2014.

\bibitem{P.8} Y.~Malyshkin and E.~Paquette. \newblock The power of choice combined with preferential attachment. \newblock {\em Electronic Communications in Probability}, 19, 2014.

\bibitem{P.7} Y.~Malyshkin and E.~Paquette. \newblock The power of choice over preferential attachment. \newblock {\em ALEA}, 12(2):903--915, 2015.

\bibitem{P.9} R.~Pemantle et~al. \newblock A survey of random processes with reinforcement. \newblock {\em Probab. Surv}, 4(0):1--79, 2007.

\bibitem{P.28} H.~Robbins and S.~Monro. \newblock A stochastic approximation method. \newblock {\em The Annals of Mathematical Statistics}, pages 400--407, 1951.
\end{thebibliography}
\end{document}